\documentclass[10pt]{article}
\usepackage{graphicx}
\usepackage{amsmath,amssymb,amsthm,amsfonts}
\usepackage{amssymb}

\newtheorem{theorem}{Theorem}[section]

\newtheorem{definition}{Definition}[section]
 \newtheorem{lemma}{Lemma}[section]
 
 \newtheorem{remark}{Remark}[section]

\numberwithin{equation}{section}

\theoremstyle{definition}

\theoremstyle{remark}

\numberwithin{equation}{section}

\DeclareMathSymbol{\C}{\mathalpha}{AMSb}{"43}

\newcommand{\bsub}{\begin{subequations}}
\newcommand{\esub}{\end{subequations}$\!$}

\begin{document}

\title{ Linear Response Theory for Nonlinear Stochastic Differential Equations with $\alpha$-stable L\'{e}vy Noises}

\author{ Qi Zhang\thanks{Department of Applied Mathematics, Illinois Institute of Technology, Email:qzhang82@hawk.iit.edu}, Jinqiao Duan\thanks{Department of Applied Mathematics, Illinois Institute of Technology, Email:duan@iit.edu}}


\smallbreak \maketitle

\begin{abstract}
We consider a nonlinear stochastic differential equation driven by an $\alpha$-stable L\'{e}vy process ($1<\alpha<2$). We first obtain some regularity results for the probability density of its invariant measure via establishing the a priori estimate of the corresponding stationary Fokker-Planck equation. Then by the a priori estimate of Kolmogorov backward equations and the perturbation property of Markov semigroup, we derive the response function and generalize the famous linear response theory in nonequilibrium statistical mechanics to non-Gaussian stochastic dynamic systems.
\end{abstract}

\vskip 0.2truein


\vskip 0.2truein

\section{Introduction}
We consider a stochastic dynamical system described by a stochastic differential equation (SDE) on $\mathbb{R}^n$:
\begin{equation}\label{Lequ}
 dX_t = b(X_{t})dt + \sigma(X_{t^{-}})dL_t,
\end{equation}
where $b(x):\mathbb{R}^n \rightarrow \mathbb{R}^n$ and $\sigma(x)=(\sigma_{ij}(x))_{n\times n}:\mathbb{R}^n \rightarrow  \mathbb{R}^{n \times n}$ are Borel measurable functions, and $L_t$ is a $n$-dimensional $\alpha$-stable L\'{e}vy process on the filtered probability space $(\Omega,\mathcal{F},(\mathcal{F}_t)_{t\geq 0},\mathbb{P})$ with $1<\alpha<2$. Assume the SDE (\ref{Lequ}) is ergodic with the unique invariant distribution $\mu$. If $\mu$ be the initial distribution of $(X_t)_{t\geq 0}$, then $(X_t)_{t\geq 0}$ is a stationary Markov process. In physics, $(X_t)_{t\geq 0}$ is a stationary Markov process means that the corresponding stochastic dynamical system is in a steady state (including the equilibrium state and nonequilibrium steady state).

In recent years, stochastic dynamical systems with L\'{e}vy noises have attracted a lot of attention in many areas, including  modeling the DNA-target search for binding sites \cite{SNG2011}, and active transport within cells \cite{LVS2015}. Signatures of L\'{e}vy noise and anomalous transport have
been found ubiquitous in nature.
In anomalous transport, the particle undergoing L\'{e}vy super diffusion is performing motion with random jumps following a power-law distribution. In complex physical systems, many experimental works demonstrate that the distribution of various fluctuations is also heavy-tailed typical of L\'{e}vy-type distributions \cite{KRS2008}.

In statistical physics, the fluctuation-dissipation theorem holds for dissipative systems near the equilibrium states, and is a useful tool in investigation of physical properties of systems at thermodynamic equilibrium. It connects the energy dissipation in an irreversible process to the thermal fluctuation in equilibrium through suitable correlation functions.
This is explained in the following Langevin equation
\begin{equation*}
    \dot{x}=v, \quad m\dot{v}= - \gamma v+\sqrt{2D} \dot{W}(t),
\end{equation*}
where $x$ is the position of one particle with velocity $v$, $\gamma >0$ is friction coefﬁcient, $D$ is diffusion coefficient, and $\dot{W}(t)$ is a Gaussian white noise which could be understood as a random force. The fluctuation-dissipation  theorem \cite{Kubo196,Z2001,ZQQ2012} provides a precise connection between the dissipation term $\gamma$ and the fluctuation term $\sqrt{2D}\dot{W}(t)$, such that the covariance satisfies
\begin{equation}\label{FDTG}
    D\mathbb{E}(\dot{W}(t)\dot{W}(s))= k_BT\gamma \delta(t-s),
\end{equation}
where $k_B$ is the Boltzmann constant and $T$ is the absolute temperature, leading to the important Einstein relation $D=k_B T\gamma$. By virtue of the fluctuation-dissipation theorem, measurable macroscopic physical quantities like the average kinetic energy or susceptibilities can be related to correlation functions of spontaneous flﬂuctuations.

On the other hand, many works \cite{CG2008,CJ2020,Dy2012,PJP2009,ZQQ2012} indicate that the classical fluctuation-dissipation theorem is a special case of a more general fluctuation relation, and it is still hold for many different non-equilibrium systems. This fluctuation relation can be described by the linear response theory.
The linear response theory can be viewed as a generalization of the well-known fluctuation-dissipation theorem when systems near steady states. Moreover, it is valid under more general conditions with many variables, including positions, velocities, concentrations, and order parameters.

Let us review the linear response theory introduced in \cite{PJP2009,ZQQ2012}. The theorem only requires that the system $(X_t)_{t\geq 0}$ is a Markov process with a invariant measure. Consider a stochastic dynamics system in steady state, i.e. the initial distribution of $(X_t)_{t\geq 0}$ is a invariant measure of Markov semigroup, and $(X_t)_{t\geq 0}$ is a stationary Markov process. And an small external perturbation is applied to the system. Let $X^{ F}_t$ be the perturbed process. Given an arbitrary observable $O(x)$, the response (evaluated to first order in the perturbation) can be written as
\begin{equation}\label{resf}
     \mathbb{E}O(X^{F}_t)-\mathbb{E}O(X_{0})\approx \int_0^t \mathcal{R}_{O}(t-s) F(s) ds,
\end{equation}
where $\mathcal{R}_{O}(t)$ is the time-dependent susceptibility of variable $O$, and is called response function in linear response theory.
The linear response theory states the relationship between the response function and a cross-correlation function
\begin{equation}\label{resrelation}
    \mathcal{R}_{O}(t) =\frac{d}{dt}\mathbb{E}(O(X^F_t)U(X_{0})),
\end{equation}
where $U(x)$ is the variable conjugate with respect to the perturbation. The linear response theory reveals the susceptibility of every observable when the stochastic dynamic system closed to the steady state and then the response to an small time-dependent perturbation.

The mathematical formulation of linear response theory for dissipative stochastic dynamical systems are considered in many works. Dembo and Deuschel \cite{DD2010} have developed the mathematical theory of linear response theory for homogenous Markov processes based on the methods of strongly continuous semigroups and Dirichlet forms. Chen and Jia \cite{CJ2020} provided rigorous mathematical proofs of linear response theory and the Agarwal-type fluctuation-dissipation theorem for a stochastic differential equation deriven by a Brownian motion with unbounded coefficients and a general perturbation.
In recent years, some physicists begin to consider the linear response theory to stochastic differential equations driven by L\'{e}vy processes. In \cite{Dy2012}, the anthers considered the linear stochastic differential equations with stable L\'{e}vy noise and constant coefficients and established the linear response theory. The linear response theory and Onsagers fluctuation theory to linear stochastic differential equations driven by a Gaussian noise and a Cauchy noise have been studied in \cite{KDG2018,KESG2013}.

In this present paper, we study the linear response theory for nonlinear stochastic differential equations driven by an $\alpha$-stable L\'{e}vy process ($1<\alpha<2$) with rigorous mathematical formulation. We assume that there is a perturbation $F(t)K(x)$ to the drift term, where $F(t)\in L^{\infty}(\mathbb{R}^{+})$, $ K(x), \text{div}(K(x)) \in  L^{\infty}(\mathbb{R}^n)$, and $\|F\|_{L^{\infty}} \ll 1$. Under the external perturbation $F(t)K(x)$, the perturbed process $X^{F}_t$ is the solution of following stochastic differential equation
\begin{equation}\label{LequU}
  \left\{
   \begin{aligned}
   & dX^{ F}_t = (b(X^{ F}_t)+ F(t)K(X^{ F}_{t}))dt + \sigma(X^{ F}_{t^{-}})dL_t,\\
   & X^{ F}_0=X_{0},
   \end{aligned}
   \right.
\end{equation}
where the distribution of $X_{0}$ is an invariant measure of the SDE (\ref{Lequ}). We prove that the fluctuation relation (\ref{resrelation}) is true for the SDE (\ref{Lequ}).

The main tools to establish the linear response theory is the Markov semigroup and Kolmogorov backward equations. We obtain the perturbation property of the corresponding Markov semigroup by the a priori estimate of the Kolmogorov backward equation. Then we establish the linear response theory and the Agarwal-type fluctuation-dissipation theorem for SDE (\ref{Lequ}. In the present paper, combine with nonlocal heat kernel estimates, we also prove a new ergodicity result of SDE (\ref{Lequ}) by the Bogoliubov-Krylov argument. Moreover, we derive a new form of Fokker-Planck equation associated with the SDE (\ref{Lequ}), and establish regularity for the density of invariant measure of SDE (\ref{Lequ}) via to establish the aprior estimate for the corresponding stationary Fokker-Planck equation.

This paper is organized as follows. In section 2, we revisit some basic notation and definitions of the SDE driven by an $\alpha$-stable L\'evy process, and introduce some well-posedness and ergodicity results for the SDE (\ref{Lequ}) from \cite{CZZ2017,XZ2020}. In section 3, we prove the ergodicity of SDE (\ref{Lequ}) and the existence and uniqueness of invariant measure by the Bogoliubov-Krylov argument. Then we derive the Fokker-Planck equation associated with the SDE (\ref{Lequ}) and establish regularity results for the invariant measure. In section 4, we obtain the response function, and establish the linear response theory as Theorem \ref{LRThm}. In addition, the Agarwal-type fluctuation-dissipation theorem for SDE (\ref{Lequ}) is also obtained as Theorem \ref{AFDT}. The paper ends with some summary and discussions in section 5.

\section{Preliminaries}
In this section, we recall some basic notation and definitions. After making some assumptions, we introduce a well-posedness result of SDE (\ref{Lequ}) and the corresponding Kolomogrov equation. In the end, we recall some basic notions about suitable invariant measure and ergodicity, and make the dissipativity assumption for SDE (\ref{Lequ}).
\subsection{Basic notations and definitions}
We first introduce some spaces and notations.
For $p\in [1,\infty]$, let $L^p(\mathbb{R}^n)$ be the usual Lebesgue space of all Borel functions on $\mathbb{R}^n$ with $L^p$ norm.
For $0<\alpha \leq 2$ and $1<p<\infty$, let $H^{\alpha}_p(\mathbb{R}^n)$ be the usual Bessel potential space with the norm
\begin{equation*}
  \|f\|_{H^{\alpha}_p}=\|((I-\Delta)^{\frac{\alpha}{2}})^{-1}f\|_{L^p},
\end{equation*}
where $(I-\Delta)^{\frac{\alpha}{2}}$ and $\Delta^{\frac{\alpha}{2}}$ are defined by
\begin{equation*}
  (I-\Delta)^{\frac{\alpha}{2}}f:=\mathcal{F}^{-1}((1+|\cdot|^2)^{\frac{\alpha}{2}}\mathcal{F}f), \quad (-\Delta)^{\frac{\alpha}{2}}f:=\mathcal{F}^{-1}(|\cdot|^{\alpha}\mathcal{F}f).
\end{equation*}
When $p\in [1,\infty]$ and $\alpha \in \mathbb{N}$, $H^{\alpha}_p(\mathbb{R}^n)$ is denoted for usual Sobolev space with the norm
\begin{equation*}
  \|f\|_{H^{\alpha}_p}=(\sum_{|\theta|\leq \alpha}\|\partial^{\theta}f \|^p_{L^p})^{\frac{1}{p}}.
\end{equation*}
We recall the following Sobolev embedding. Let $0\leq \theta \leq\alpha\leq 2$, and let $1 \leq p \leq q \leq \infty$ such that $\theta-\frac{n}{q}<\alpha-\frac{n}{p}.$
Then $H^{\alpha}_p(\mathbb{R}^n) \hookrightarrow H^{\theta}_q(\mathbb{R}^n)$.

Let $X$ denote a Banach space. Let $L(X)$ be the Banach space of linear bounded operators from $X$ to $X$. For every $p \in [1,\infty]$ and $0\leq s<t \leq \infty$, the space $L^p(s,t;X)$ consists of all strongly measurable $u:[s,t]\rightarrow X$ with
\begin{equation*}
  \|u\|_{L^p(s,t;X)}:=\left(\int^t_s \|u(r)\|^p_X dr \right)^{\frac{1}{p}}<\infty
\end{equation*}
for $1\leq p \leq \infty$, and
\begin{equation*}
  \|u\|_{L^{\infty}(s,t;X)}:=\text{esssup}_{s\leq r \leq t}\|u(r)\|_X <\infty.
\end{equation*}

Now we recall some basic facts for $\alpha$-stable L\'{e}vy processes from \cite{A2004,D2015}. Let $(\Omega,\mathcal{F},\mathbb{P})$ be a filtered probability space satisfying the usual conditions.
Consider the $n$ dimensional $\alpha$-stable L\'{e}vy process $L_t$ on $(\Omega,\mathcal{F},(\mathcal{F}_t)_{t\geq 0},\mathbb{P})$ with $1<\alpha<2$.
The characteristic function of $L_t$ is
\begin{equation*}
   \varphi_{L_t}(\xi)= \mathbb{E}\exp(i\langle \xi, L_t \rangle) = e^{t|\xi|^{\alpha}}.
\end{equation*}

For the $\alpha$-stable L\'{e}vy process $L_t$, the corresponding L\'{e}vy measure $\nu(dy) = \frac{c_{\alpha}dy}{|y|^{d+\alpha}} $. We denote by $N^L(dt,dy)$ the Poisson random measure associated to the pure jump-process $\Delta L_t =L_{t}-L_{t^-}$ such that $\mathbb{E}[N^L(dt,dy)] = dt\nu(dy)$, which is defined as
\begin{equation*}
    N^{L}(t,B)(\omega):= \sharp \{s\in [0,t]:\Delta L_{s}(\omega) \in B\}, \quad t\geq 0, B\in \mathcal{B}(\mathbb{R}^{n}\backslash \{0\}).
\end{equation*}
And the corresponding compensated Poisson random measure $\widetilde{N}^L$ is defined as
\begin{equation*}
   \widetilde{N}^{L}(dt,dy)=N^{L}(dt,dy)-dt\nu(dy) .
\end{equation*}
Then by L\'{e}vy-It\^{o} decomposition theorem, we have following path-wise description of $L_t$
\begin{equation*}
    L_t = \int_0^t \int_{0<|y|<1}y\widetilde{N}^L(dt,dy)+ \int_0^t \int_{|y|\geq 1}yN^L(dt,dy).
\end{equation*}

\subsection{SDE driven by $\alpha$-stable L\'evy process}
Consider the following stochastic differential equation on $\mathbb{R}^n$:
\begin{equation}\label{Lequt}
 dX_t = \hat{b}(t,X_{t})dt + \sigma(X_{t^{-}})dL_t,
\end{equation}
where $\hat{b}(t,x):\mathbb{R}^{+}\times\mathbb{R}^n \rightarrow \mathbb{R}^n$ and $\sigma(x)=(\sigma_{ij}(x))_{n\times n}:\mathbb{R}^n \rightarrow  \mathbb{R}^{n \times n}$ are Borel measurable functions, and $L_t$ is a $n$-dimensional $\alpha$-stable Levy process on $(\Omega,\mathcal{F},(\mathcal{F}_t)_{t\geq 0},\mathbb{P})$ with $1<\alpha<2$. Note that here the drift term $\hat{b}(t,x)$ is dependent of $t$ and $x$, and the form of (\ref{Lequt}) includes the SDE (\ref{Lequ}) and the perturbed SDE (\ref{LequU}). So the following results in this subsection are hold for these two SDEs.

The SDE (\ref{Lequt}) is equivalent to
\begin{equation}\label{Lequu}
 dX_t = \hat{b}(t,X_{t})dt + \int_{|y|<1}\sigma(X_{t^{-}})y\widetilde{N}^L(dt,dy)+\int_{|y|\geq 1}\sigma(X_{t^{-}})yN^L(dt,dy).
\end{equation}

We make the following assumptions on the drift coefficient $b$ and the diffusion coefficient $\sigma$.\\
{\bf (A)}(H\"{o}lder continuous) For all $(t,x) \in  \mathbb{R}^{+}\times\mathbb{R}^n$, there are $c_0 >0$ and $\beta \in (1-\alpha/2, 1)$ such that
\begin{equation*}
    |\hat{b}(t,x)-\hat{b}(t,y)|\leq c_0 |x-y|^{\beta}, \quad \|\sigma(x)-\sigma(y)\| \leq c_0 |x-y|^{\beta}.
\end{equation*}
Here and below, $\|\cdot\|$ denotes the Hilbert-Schmidt norm of a matrix, and $|\cdot|$ denotes the Euclidean norm.\\
{\bf (B)}(Uniform ellipticity) There exists a constant $\Lambda>0$ such that for all $x \in \mathbb{R}^n$,
\begin{equation*}
    \Lambda^{-1}|\xi| \leq |\sigma(x)\xi)| \leq \Lambda|\xi|, \quad \forall \xi\in \mathbb{R}^n.
\end{equation*}
{\bf (C)}(Uniform boundedness) There exists a constant $c_0>0$ such that for all $(t,x) \in  \mathbb{R}^{+}\times\mathbb{R}^n$,
\begin{equation*}
  |\text{div}(\hat{b}(x)|,|\hat{b}(t,x)|,|\nabla \sigma(x)|,|\nabla \sigma^{-1}(x)|< c_0.
\end{equation*}
The following well-posedness result of the SDE (\ref{Lequt}) is proved by Zhen, Zhang and Zhao (see \cite{CZZ2017}, Theorem 1.1).

\begin{theorem}\label{ExD}
Suppose that $({\bf A})$-$({\bf C})$ hold. Then for each $X_0=x \in \mathbb{R}^n$, the SDE (\ref{Lequt}) admits a unique strong solution $X_{t}$.
\end{theorem}

The stochastic process $(X_t)_{t\geq 0}$ is a Markov process with a Markov transition kernel
\begin{equation*}
    \pi_{s,t}(x,B):=\mathbb{P}(X_{t}\in B|X_{s}=x), \quad 0\leq s \leq t, x\in \mathbb{R}^n, B\in \mathcal{B}(\mathbb{R}^n).
\end{equation*}
We denote by $(P_{s,t})_{t-s\geq 0}$ the associated Feller semigroup of $(X_t)_{t\geq 0}$, i.e.
\begin{equation*}
    P_{s,t}f(x):=\int_{\mathbb{R}^n}f(y)\pi_{s,t}(x,dy)=\mathbb{E}(f(X_t)|X_s =x), \quad 0\leq s\leq t, \ x\in\mathbb{R}^n,
\end{equation*}
where $f\in \mathcal{B}_b(\mathbb{R}^n)$. If the Markov process is time-homogeneous, we denote $P_{s,t}=P_{t-s}$ for all $0\leq s\leq t$.  If the transition probability densities $p(s,x;t,y)$ exists, then
\begin{equation*}
    P_{s,t}f(x):=\int_{\mathbb{R}^n}f(y)p(s,x;t,y)dy.
\end{equation*}

The generator $A(s)$ of $P_{s,t}$ is the following integro-differential operator
\begin{align*}
    A(s)u(x) := & \hat{b}(s,x) \cdot \nabla u + \int_{\mathbb{R}^{n}\setminus \{0\}} [u(x+\sigma(x)y)-u(x)] \nu(dy),
\end{align*}
where $u \in Dom(A)\subset L^2(\mathbb{R}^n) $.

By It\^{o}'s formula, for each $f\in \mathcal{B}_b(\mathbb{R}^n)$ and $0\leq s\leq t$, the function $u(t,x):=\mathbb{E}[f(X_t)|B_s =x]$ satisfies the following Kolmogorov backward equation
\begin{equation}\label{backward}
  \left\{
   \begin{aligned}
   & \partial_s u(s,x) = -A_x(s)u(s,x),\\
   & u(t,x)=f(x).
   \end{aligned}
   \right.
\end{equation}
for all $s\in [0,t]$. Moreover, the transition density $p(s,x;t,y)$ of Markov process $(X_t)_{t\geq 0}$ is the fundamental solution of following Kolmogorov backward equation
\begin{equation}\label{Kbequ}
  \left\{
   \begin{aligned}
   & \partial_s u(s,x) = -A_x(s) u(s,x),\\
   & u(t,x)=\delta (x-y).
   \end{aligned}
   \right.
\end{equation}
 Here and below, $A_x$ denotes the operator $A$ act on functions of $x$.

Furthermore, for $0 \leq s < t\leq T$, if the probability density of $X_s$ is $p_s(x)$, then the probability density $p_t$ of the Markov process $(X_t)_{t\geq 0}$ is the solution of the following Kolmogorov forward equation, or Fokker-Planck equation
\begin{equation}\label{FKequ}
  \left\{
   \begin{aligned}
   & \partial_t u(t,x) = A_x^{\ast}(t) u(t,x),\\
   & u(s,x)=p_s(x),
   \end{aligned}
   \right.
\end{equation}
where the operator $A^{\ast}(t)$ is the adjoint operator of $A$ deﬁned through $\langle Af,g \rangle_{L^2}=\langle f,A^{\ast}g \rangle_{L^2}$. And the probability density $p_t$ of $X_t$ is given by
\begin{equation*}
   p_{t}(y)=\int_{\mathbb{R}^n}p_s(x)p(s,x;t,y)dx.
\end{equation*}

Moreover, since $b(t,x)$ is uniformly bounded, from \cite{CZ2018}, Theorem 1.5, the transition probability densities $p(s,x;t,y)$ of $X_t$ exists, and it enjoys the following estimates.
\begin{theorem}\label{tpdE}
Under ${\bf (A)}$-${\bf (C)}$, there is a unique continuous function $p(s,x;t,y)$ satisfying (\ref{Kbequ}), i.e. the transition probability densities $p(s,x;t,y)$ of $X_t$ exists. Moreover, $p(s,x;t,y)$  enjoys the following properties.\\
(i) (Two-sides estimate) For every $T > 0$, there are two positive constants $c_1,c_2$ such that for $0\leq s < t \leq T$ and $x,y\in \mathbb{R}^n$,
\begin{equation}\label{2sidesE}
    c_1 (t-s)((t-s)^{\frac{1}{\alpha}}+|x-y|)^{-n-\alpha}\leq p(s,x;t,y) \leq  c_2 (t-s)((t-s)^{\frac{1}{\alpha}}+|x-y|)^{-n-\alpha}.
\end{equation}
(ii) (Gradient estimate) For every $T > 0$, there exists a positive constant $c_3$ such that for $0\leq s < t \leq T$ and $x,y\in \mathbb{R}^n$,
\begin{equation}\label{GraE}
    |\nabla_x p(s,x;t,y)|\leq  c_3 (t-s)^{-\frac{1}{\alpha}}p(s,x;t,y).
\end{equation}
(iii) (Fractional derivative estimate) For every $\theta \in [0,\alpha)$, there exists a positive constant $c_4$ such that for $0\leq s < t \leq T$ and $x,y\in \mathbb{R}^n$,
\begin{equation}\label{FDE}
    |\Delta^{\theta/2}_x p(s,x;t,y)|\leq  c_4 (t-s)^{1-\frac{\theta}{\alpha}}((t-s)^{\frac{1}{\alpha}}+|x-y|)^{-n-\alpha}.
\end{equation}
(vi) (Continuity) For every bounded and uniformly continuous function $f(x)$, we have
\begin{equation}\label{Con}
    \lim_{|t-s|\rightarrow 0} \|\int_{\mathbb{R}^n}p(s,x;t,y)f(y)dy-f(x)\|_{L^{\infty}}=0.
\end{equation}
\end{theorem}

Now from above estimates of the transition probability density, we have the following results of  solvability and
regularity of corresponding Kolomogrov equation.
By Minkowski's inequality for integral, the following result is a direct consequence of two-sides estimate (\ref{2sidesE}) and fractional derivative estimate (\ref{FDE}).

\begin{lemma}\label{EXuBack}
Assume that condition $({\bf A})$-$({\bf C})$ hold. Assume $f(x)\in L^p$ with some $1\leq p \leq \infty$. Let $\theta \in (0,\alpha)$. Then for $0\leq s \leq t$, the function $u(s,t):=P_{s,t}f(x) \in H^{\theta}_p$ is the unique solution to the Kolmogorov backward equation
\begin{equation*}
  \left\{
   \begin{aligned}
   & \partial_s u(s,x) = -A(s)u(s,x),\quad (s,x)\in [0,t)\times\mathbb{R}^n\\
   & u(t,x)=f(x),\quad x\in \mathbb{R}^n,
   \end{aligned}
   \right.
\end{equation*}
where $A(s)$ is the generator of SDE (\ref{Lequt}).
Moreover, there is a constant $C>0$ such that for all $0\leq s \leq t$,
\begin{equation*}
    \|u(s,x)\|_{L^p} \leq C \|f\|_{L^p}
\end{equation*}
and
\begin{equation*}
    \| u(s,x)\|_{H^{\theta}_p} \leq C (t-s)^{-\frac{\theta}{\alpha}}\|f\|_{L^p}.
\end{equation*}
\end{lemma}

\begin{proof}
From Theorem \ref{tpdE}, the unique solution $u(s,x)$ is given by
\begin{equation*}
    u(s,x)=\int_{\mathbb{R}^n} p(s,x,t,y)f(y)dy.
\end{equation*}
The two-sides estimate (\ref{2sidesE}) yields that
\begin{align*}
    |u(s,x)| & \lesssim \int_{\mathbb{R}^n}(t-s)[(t-s)^{\frac{1}{\alpha}}+|y|]^{-n-\alpha}|f(x-y)|dy
\end{align*}
Then for $1 \leq p<\infty$, by Minkowski's inequality for integral, we have
\begin{align*}
    \|u(s,x)\|_{L^p}& \lesssim \int_{\mathbb{R}^n}(t-s)[(t-s)^{\frac{1}{\alpha}}+|y|]^{-n-\alpha}\|f(y)\|_{L^p}dy  \\
    & \lesssim  \|f\|_{L^p}\int_{\mathbb{R}^n}(t-s)[(t-s)^{\frac{1}{\alpha}}+|y|]^{-n-\alpha}dy \\
    & \leq C \|f\|_{L^p}.
\end{align*}
This estimate is obvious when $p=\infty$.

By the fractional derivative estimate (\ref{GraE}) and two-sides estimate (\ref{2sidesE}), we have
\begin{align*}
    |\Delta^{\theta/2}_x u(s,x)| & \leq \int_{\mathbb{R}^n} |\Delta^{\theta/2}_x p(s,x;t,y)||f(y)|dy \\
                    & \lesssim (t-s)^{-\frac{\theta}{\alpha}}\int_{\mathbb{R}^n}  p(s,x;t,y)|f(y)|dy.
\end{align*}
Then we get
\begin{equation*}
    \|\Delta^{\theta/2} u(s,x)\|_{L^p}\leq C (t-s)^{-\frac{\theta}{\alpha}} \|f\|_{L^p}.
\end{equation*}
The proof is complete.
\end{proof}

We now consider the following nonlocal parabolic equation corresponding to SDE (\ref{Lequt}):
\begin{equation}\label{bacdK}
  \left\{
   \begin{aligned}
   & \partial_s w(s,x) = -A(s)w(s,x)+\lambda w(s,x)- g(s,x), \quad (s,x)\in [0,t]\times\mathbb{R}^n,\\
   & w(t,x)=0, \quad x\in \mathbb{R}^n
   \end{aligned}
   \right.
\end{equation}
where $A(s)$ is the generator of SDE (\ref{Lequt}), $\lambda\geq 0$.

As in proof of Theorem 4.5 in \cite{XZ2020}, by fractional derivative estimate (\ref{GraE}), two- sides estimate (\ref{2sidesE}) and Young’s convolution inequality, we have the following solvability and $L^p$-estimate of (\ref{bacdK}).

\begin{lemma}\label{Backe}
Assume that condition $({\bf A})$-$({\bf C})$ hold. Let $p,q\in (1,\infty)$ and $p'\in [p,\infty]$, $q'\in [q,\infty]$, $\vartheta \in [1,\alpha)$ with
\begin{equation*}
    \frac{n}{p}+\frac{\alpha}{q}<\alpha -\vartheta +\frac{n}{p'}+\frac{\alpha}{q'}.
\end{equation*}
Then for every $t>0$, $g\in L^q((0,t);L^p)$, there are constants $ c>1$ and unique mild solution $u(s,x)=\int^t_s P_{s,v}g(v,x)dv$ to (\ref{bacdK}) such that for all $\lambda \geq 0$ and $s\in [0,t]$,
\begin{equation*}
   (1\vee\lambda)^{\frac{1}{\alpha}(\alpha-\vartheta+\frac{n}{p'}+\frac{\alpha}{q'}-\frac{n}{p}-\frac{\alpha}{q})} \|w\|_{L^{q'}(s,t;H^{\vartheta}_{p'})}\leq c\|g\|_{L^q(s,t;L^p)}.
\end{equation*}
\end{lemma}

\subsection{Steady states, invariant measures and ergodicity}
We recall some basic notions about the invariant measure and ergodicity. Now we assume that the drift term $b$ and diffusion term $\sigma$ of SDE (\ref{Lequ}) is independent of time $t$. Thus the solution $(X_t)_{t\geq 0}$ is a homogeneous Markov process with Markov Feller semigroup $P_t$.

\begin{definition}
A probability measure $\mu$ on $(\mathbb{R}^{n}, \mathcal{B}(\mathbb{R}^n))$ is said to be an invariant measure under Markov semigroup $P_t$ if it satisfies
\begin{equation*}
    \int_{\mathbb{R}^n} Af(x)\mu(dx)=0, \ \forall f\in C^{\infty}_0(\mathbb{R}^n),
\end{equation*}
where $A$ is the generator of Markov semigroup $P_t$.
\end{definition}

This means $\int_{\mathbb{R}^n}(P_tf)(x)\mu(dx)=\int_{\mathbb{R}^n}f(x)\mu(dx)$ for all $t\geq 0$ and $f\in C^{\infty}_0(\mathbb{R}^n)$.
If the invariant measure $\mu$ has probability density $p_{ss}(x)$, then $p_{ss}(x)$ is a solution of stationary Fokker-Planck equation (\ref{FKequ}), i.e. $A^{\ast}p_{ss}(x)=0,$ or $P^{\ast}_{t}p_{ss}(x)=p_{ss}(x)$ for all $t>0$. And if initial distribution of $(X_{t})_{t\geq 0}$ is above invariant measure $\mu$, then $(X_{t})_{t\geq 0}$ is a stationary Markov process, which satisfies that for every $f\in \mathcal{B}_b(\mathbb{R}^n)$ and $t>0$, $P_t f(X_{0})=f(X_{0})$.

It is known that a stochastic dynamic system $(X_{t})_{t\geq 0}$ is said in an steady state if its initial distribution is the invariant measure $\mu$ of the corresponding SDE and $(X_{t})_{t\geq 0}$ is a stationary Markov process (see \cite{BL1955}). Moreover, a steady state $(X_{t})_{t\geq 0}$ is said to be in an equilibrium state if $(X_{t})_{t\geq 0}$ is a symmetric Markov process with respect to the invariant measure $\mu$, and $(X_{t})_{t\geq 0}$ is said to be in an non-equilibrium steady state if $(X_{t})_{t\geq 0}$ is a non-symmetric Markov process with respect to the invariant measure $\mu$ (see \cite{CG2008,ZQQ2012}).

\begin{definition}\label{ergodic}
A Morkov semigroup $P_t$ is ergodic if $P_t$ admits a unique invariant probability measure $\mu$, which amounts to say that
\begin{equation*}
    \lim_{t\rightarrow \infty}\frac{1}{t}\int_0^t P_s f(x)ds = \int_{\mathbb{R}^n}f(x)\mu(dx), \quad \forall f\in \mathcal{B}_b(\mathbb{R}^n).
\end{equation*}
\end{definition}
To give the ergodicity result, we need a dissipativity assumption for drift term $b(x)$.\\
$({\bf D})$(Dissipativity) For all $x\in \mathbb{R}^n$, there exits a constant $k_1>0$, such that
\begin{equation*}
     \langle x,b(x) \rangle \leq -k_1|x|
\end{equation*}
Moreover, the constant $k_1$ satisfies
\begin{equation*}
  \sqrt{2} k_1 > \Lambda^2 \int_{0<|y|<1}|y|^2 \nu(dy)+\Lambda\int_{|y|\geq 1}|y| \nu(dy),
\end{equation*}
where $\Lambda$ is the constant in uniformly elliptic assumption $(\bf B)$.

A usual method for proving the existence of invariant measures of Markov processes is the Bogoliubov-Krylov argument, which is based on Lyapunov functions (e.g. \cite{D2006}, Theorem 7.1 and Proposition 7.10).
\begin{theorem}\label{BKthm}
Let $X_t(x_0)$ be a Morkov process with initial value $X_0=x_0 \in\mathbb{R}^n$, and $P_t$ be the corresponding Markov Feller semigroup of $X_t$. Let $V:\mathbb{R}^n\rightarrow \mathbb{R}^{+}$ be a Borel measurable function whose level sets
\begin{equation*}
  K_a:= \{x\in \mathbb{R}^n:\ V(x)\leq a\},\quad a>0,
\end{equation*}
are compact for every $a>0$. Assume that there exists $x_0\in\mathbb{R}^n$ and $C(x_0)>0$ such that
\begin{equation*}
  \mathbb{E}(V(X_t(x_0)))<C(x_0), \quad \forall t\geq 0.
\end{equation*}
Then there is an invariant measure for $P_t$.
\end{theorem}

\begin{remark}
In Theorem \ref{BKthm}, the Borel measurable function $V(x)$ is called a Lyapunov function for $P_t$.
\end{remark}
We recall the following notations for Markov Feller semigroup $P_t$.\\
$\bf{(Strong \ Feller)}$ $P_t$ has the strong Feller property if for all $f \in\mathcal{B}_b(\mathbb{R}^n)$, $P_t f\in C_b(\mathbb{R}^n)$.\\
$\bf{(Irreduciblility)}$ $P_t$ is irreducible if for each open ball $B$ and for all $t>0$, $x\in\mathbb{R}^n$, $P_t 1_{B_1}(x)>0$.\\
We have a sufficient condition for the uniqueness of invariant measure of Markov semigroup $P_t$.
\begin{lemma}
If $P_t$ is strong Feller and irreducible, then it possesses at most one invariant measure.
\end{lemma}
The following ergodic result is standard.
\begin{theorem}
If a Markov process has a unique invariant measure, then it is ergodic.
\end{theorem}

\section{Existence, uniqueness and regularity of invariant measure}
\subsection{Existence of invariant measure and ergodic property}

We show the following moment estimate of the unique strong solution to (\ref{LequU}).
\begin{lemma}\label{GEME}
Assume that $({\bf A})$-$({\bf D})$ hold. Let $X_0$ be the initial value of SDE (\ref{Lequ}) with initial distribution $\mu_0$. Suppose that $X_0$ has a finite first moment, i.e. $\mathbb{E}|X_0|<\infty$. Then for solution $(X_t)_{t \geq 0}$ to SDE (\ref{LequU}) with initial value $X_0$, there exits a positive constant $C$ such that
\begin{equation}\label{momentE}
  \mathbb{E}|X_t| \leq \mathbb{E}\sqrt{1+|X_0|^2}+C, \quad \forall t>0.
\end{equation}
\end{lemma}
\begin{proof}
Define $r(x):=\sqrt{1+|x|^2}$. Then by It\^{o}'s formula, for $\forall t>0$ we have
\begin{align*}
     & r(X_{t})-r(X_{0})\\
    =& \int_{0}^{ t} b(X_{s^{-}})\nabla r(X_{s^{-}})ds\\
     & + \int_{0}^{t} \int_{\mathbb{R}^N\setminus \{0\}} r(X_{s^{-}}+\sigma(X_{s^{-}})y)-rh(X_{s^{-}}) \widetilde{N}^L(ds,dy)\\
     & +\int_{0}^{t} \int_{\mathbb{R}^N\setminus \{0\}}[r(X_{s^{-}}+\sigma(X_{s^{-}})y)-r(X_{s^{-}})-1_{B_1(0)}(y) \sigma(X_{s^{-}})y\nabla r(X_{s^{-}})]\nu(dy)ds.
\end{align*}
Take expectation for two sides, and we have
\begin{equation*}
  \mathbb{E}r(X_{t})=\mathbb{E}r(X_{0})+ \mathbb{E}\int_{0}^{t}Ar(X_s)ds.
\end{equation*}
The stochastic Fubini theorem implies that
\begin{equation*}
  \mathbb{E}r(X_{t})=\mathbb{E}r(X_{0})+\int_{0}^{ t}\mathbb{E}Ar(X_s)ds.
\end{equation*}
By the dissipativity assumption $(\bf{D})$, we have
\begin{align}\label{drift}
  b(x)\cdot\nabla r(x)=(b(x)\cdot x)(1+|x|^2)^{-\frac{1}{2}}\leq & -k_1 |x|(1+|x|^2)^{-\frac{1}{2}} \leq -\frac{k_1}{\sqrt{2}}1_{\{ |x|\geq 1\} },
\end{align}
Note that
\begin{equation*}
  |r(x+y)-r(x)|\leq |y|\int_0^1 |\nabla r(x+sy)|ds \leq \frac{|y|}{2},
\end{equation*}
and
\begin{equation*}
  r(x+y)-r(x)-y\cdot \nabla r(x) \leq \frac{|y|^2}{2}.
\end{equation*}
We have
\begin{align}\label{intergral}
    & \int_{\mathbb{R}^{n}\setminus \{0\}} [r(x+\sigma(x)y)-r(x)-1_{B_{1}(0)}(y)\sigma(x)y \cdot \nabla r(x)] \nu(dy) \nonumber \\
  \leq & \frac{1}{2}\int_{0<|y|<1}|\sigma(x)y|^2 \nu(dy)+\frac{1}{2}\int_{|y|\geq 1}|\sigma(x)y| \nu(dy) \nonumber  \\
  \leq & \frac{1}{2}\left( \Lambda^2 \int_{0<|y|<1}|y|^2 \nu(dy)+\Lambda\int_{|y|\geq 1}|y| \nu(dy)\right) < \infty.
\end{align}
Denote $C_1=\Lambda^2 \int_{0<|y|<1}|y|^2 \nu(dy)+\Lambda\int_{|y|\geq 1}|y| \nu(dy)>0$. Combining with (\ref{drift}) and (\ref{intergral}), we have
\begin{align*}
 Ar(x) =& b(x) \cdot \nabla r(x) + \int_{\mathbb{R}^{n}\setminus \{0\}} [r(x+\sigma(x)y)-r(x)-1_{B_{1}(0)}(y)\sigma(x)y \cdot \nabla r(x)] \nu(dy) \\
 \leq & -\frac{k_1}{\sqrt{2}}1_{\{|x|\geq 1\} }+ \frac{C_1}{2},
\end{align*}
Then we get
\begin{align*}
  \mathbb{E}r(X_t) = & \mathbb{E}r(X_{0})+\int_{0}^{ t}\mathbb{E}Ar(X_s)ds \\
  \leq & \mathbb{E}r(X_{0})-\frac{k_1}{\sqrt{2}}\int_{0}^{t} P(|X_s|\geq 1)ds + \frac{t C_1}{2} \\
  = & \mathbb{E}r(X_{0})+(\frac{C_1}{2}-\frac{k_1}{\sqrt{2}})t + \frac{k}{\sqrt{2}}\int_{0}^{t}P(|X_s|<1)ds.
\end{align*}
By two-sides estimate (\ref{2sidesE}) and Fubini theorem, there exits a constant $C_2$ such that
\begin{align*}
  P(|X_s|<1)= & \int_{|y|<1}\int_{\mathbb{R}^n}p(0,x;s,y)\mu_0(dx)dy \\
   < & c_2\int_{\mathbb{R}^n}\int_{|y|<1}s((s)^{\frac{1}{\alpha}}+|x-y|)^{-n-\alpha}dy\mu_0(dx)\\
   < & \frac{c_2 \pi^{n/2}}{\Gamma(\frac{n}{2}+1)} s^{-\frac{n}{\alpha}}.
\end{align*}
Then for each $t>1$, we have
\begin{equation*}
  \int_{0}^{t}P(|X_s|<1)ds\leq 1+\int_{1}^{t}P(|X_s|<1)ds=1-\frac{\alpha}{\alpha-n}(t^{1-\frac{n}{\alpha}}-1) \frac{c_2 \pi^{n/2}}{\Gamma(\frac{n}{2}+1)}.
\end{equation*}
By dissipativity condition $\bf {(D})$, $-\sqrt{2} k_1+C_1 <0$. Thus there exists a constant $C>0$ such that
\begin{equation*}
  \mathbb{E}|X_t| \leq \mathbb{E}r(X_t) \leq \mathbb{E}r(X_0)+C, \quad \forall t>0.
\end{equation*}
The proof is complete.
\end{proof}

Now we prove the following ergodic result for SDE (\ref{Lequ}).
\begin{lemma}\label{1orderM}
Assume that $({\bf A})$-$({\bf D})$ hold. Then there exists a unique invariant measure $\mu$ for SDE (\ref{Lequ}), and the SDE (\ref{Lequ}) is ergodic. Moreover,  if $X_{ss}$ is a random variable with invariant distribution $\mu$, then $X_{ss}$ has finite first moment, i.e.  $\int_{\mathbb{R}^n}|x|\mu(dx)<\infty$.
\end{lemma}

\begin{proof}
Assume $X_0=x_0$ for some $ x_0 \in \mathbb{R}^n$. Then Lemma \ref{GEME} implies $\mathbb{E}|X_t|<|x_0|+C$ for each $t>0$. If we choose $V(x)=|x|$ as the Lyapunov function
for $P_t$, then by Theorem \ref{BKthm}, there is an invariant measure for the Markov Feller semigroup $P_t$.

Now we prove the uniqueness of invariant measure. By two-sides estimates (\ref{2sidesE}) for transition probability density $p(s,x;t,y)$, the Markov semigroup $P_t$ of $X_t$ is irreducible. Moreover, since the transition probability density $p(s,x;t,y)$ is unique continuous, for each $f\in\mathcal{B}_b(\mathbb{R}^n)$,
\begin{equation*}
  P_t f(x)=\int_{\mathbb{R}^n}p(0,x;t,y)f(y)dy\in C_b(\mathbb{R}^n).
\end{equation*}
So the Markov semigroup $P_t$ is strong Feller. Thus there exists a unique invariant measure for SDE (\ref{Lequ}), and the SDE (\ref{Lequ}) is ergodic.

For each $m \in \mathbb{N}$, consider the bounded measurable function $x1_{B_{m}(0)}(x)\in \mathcal{B}_b(\mathbb{R}^n)$. Then from the definition of ergodic property, we have
\begin{align*}
  \mathbb{E}(|X_{ss}1_{\{|X_{ss}|<m\}}|) = & \int_{\mathbb{R}^n}|x|1_{B_{m}(0)}(x)\mu(dx) \\
  = & \lim_{t\rightarrow \infty}\frac{1}{t}\int_0^t P_s |x|1_{B_{m}(0)}(x)ds \\
  = & \lim_{t\rightarrow \infty}\frac{1}{t}\int_0^t \mathbb{E}(|X_{s}1_{\{|X_{s}|<m\}}(\omega)|) ds \\
  \leq & \sup_{t\geq 0} \mathbb{E}|X_t| \\
  < & |x_0| +C.
\end{align*}
Note that $x_0+C$ is fixed, and it is independent with $m$. Let $m\rightarrow \infty$, we obtain $\mathbb{E}|X_{ss}|<\infty$.
\end{proof}

Now we prove that the invariant measure $\mu$ has a density $p_{ss}$.

\begin{lemma}\label{density}
Suppose that $(\bf{A})$-$(\bf{C})$ holds. Then the invariant measure $\mu$ has a density $p_{ss}\in L^{p'}(\mathbb{R}^n)$ with $p'<\frac{n}{n-\alpha}$ for $n\geq 2$ and $ p'<\infty$ for $n=1$.
\end{lemma}

\begin{proof}
From Lemma \ref{Backe}, for each $f\in C^{\infty}_0(\mathbb{R}^n)$ and $T>0$, there is a unique solution $u\in L^{\infty}(0,T;H^{\vartheta}_{\infty}(\mathbb{R}^n))$ solving the following equation
\begin{equation*}
  \left\{
   \begin{aligned}
   & \partial_t u(t,x) = -Au(t,x)- f(x), \quad (t,x)\in [0,T]\times\mathbb{R}^n,\\
   & u(T,x)=0, \quad x\in \mathbb{R}^n
   \end{aligned}
   \right.
\end{equation*}
By It\^{o}'s formula, we have
\begin{align*}
  \mathbb{E}u(T,X_T)& =\int_{0}^{T}\mathbb{E}(\partial_t u(t,x)+Au(t,x))dt \\
    & = \mathbb{E} \int_{0}^{T} f(X_t)dt
\end{align*}
Then the a priori estimate implies that for all $(\frac{n}{\alpha}\vee 1)<p< \infty$,
\begin{equation*}
  \mathbb{E}\int_{0}^{T}|f(X_t)|dt \leq \|u\|_{L^{\infty}(0,T;L^{\infty}(\mathbb{R}^n))}\leq C \|f\|_{L^p}.
\end{equation*}
Since $C^{\infty}_0(\mathbb{R}^n)$ is dense in $L^p(\mathbb{R}^n)$, by the ergodic property of $X_t$, we get for all $(\frac{n}{\alpha}\vee 1)<p< \infty$,
\begin{equation*}
  |\int_{\mathbb{R}^n}f(x)\mu(dx)| \leq C \|f\|_{L^p}, \quad \forall f\in L^p(\mathbb{R}^n).
\end{equation*}
Thus $f \longmapsto\int_{\mathbb{R}^n}f(x)\mu(dx)$ is linear bounded functional on $L^p(\mathbb{R}^n)$. Then by Riesz’s representation theorem, the unique invariant measure $\mu$ has a density $p_{ss}\in p'$ with $1\leq  p'<\frac{n}{n-\alpha}$ for $n\geq 2$ and $ 1\leq p'<\infty$ for $n=1$.
\end{proof}

\subsection{The adjoint operator of generator}
In order to obtain a form of the stationary Fokker-Planck equation corresponding to the SDE (\ref{Lequ}), we need derive the adjoint operator $A^{\ast}$.

By assumption $(\bf{B})$, the diffusion coefficient $\sigma(x)$ is a invertible matrix, and $\sigma(x)$ and $\sigma^{-1}(x)$ are uniform bounded for all $x\in \mathbb{R}^n$.
After changing of variables $\sigma(x)y\rightarrow y$, we can rewrite the generator $A$ as
\begin{align}\label{generator2}
    Au(x)  = & b(x)\cdot \nabla u(x) +\int_{\mathbb{R}^{n}\setminus \{0\}}(u(x+y)-u(x))\frac{\det(\sigma^{-1}(x))}{|\sigma^{-1}(x)y|^{n+\alpha}}dy \nonumber\\
    = & b(x)\cdot \nabla u(x)+\int_{\mathbb{R}^{n}\setminus \{0\}}(u(x+y)-u(x))k(x,y) \nu(dy),
\end{align}
where
\begin{equation*}
 k(x,y)=\frac{1}{\det(\sigma(x))}\left(\frac{|y|}{|\sigma^{-1}(x)y|}\right)^{n+\alpha}.
\end{equation*}
By assumption $(\bf{B})$, $k(x,y)$ is a positive and bounded function on $\mathbb{R}^n \times \mathbb{R}^n$ satisfying
\begin{equation*}
  0<r_0 \leq k(x,y) \leq r_1, \quad k(x,y)=k(x,-y),
\end{equation*}
and for all $y,z\in \mathbb{R}^n$,
\begin{equation*}
  |k(x,y)-k(z,y)| \leq r_2 |x-z|^{\beta},
\end{equation*}
where $\beta \in (0,1)$ is same in assumption $(\bf{B})$.

Consider the generator $A$ with following form
\begin{align}
    Au(x)  = & b(x)\cdot \nabla u(x) +\int_{\mathbb{R}^{n}\setminus \{0\}}(u(x+y)-u(x))\frac{\det(\sigma^{-1}(x))}{|\sigma^{-1}(x)y|^{n+\alpha}}dy \nonumber\\
    = & b(x)\cdot \nabla u(x)+\int_{\mathbb{R}^{n}\setminus \{0\}}(u(x+y)-u(x))k(x,y) \nu(dy),
\end{align}
where
\begin{equation*}
 k(x,y)=\frac{1}{\det(\sigma(x))}\left(\frac{|y|}{|\sigma^{-1}(x)y|}\right)^{n+\alpha}.
\end{equation*}
For every $\varphi \in C_0^{\infty}(\mathbb{R}^n)$, by Fubini Theorem, we have
\begin{align}\label{innerproduct}
  \langle \varphi,Au \rangle_{L^2} = &  \int_{\mathbb{R}^n} \varphi(x)(b(x))\cdot \nabla u(x) dx \nonumber \\
     & + \int_{\mathbb{R}^{n}\setminus \{0\}}\int_{\mathbb{R}^n}\varphi(x-y) u(x)k(x-y,y)-\varphi(x)u(x)k(x,y) dx \nu(dy) \nonumber \\
     = & -\int_{\mathbb{R}^n} u(x) \text{div}[b(x)\varphi(x)] dx \nonumber \\
     & + \int_{\mathbb{R}^n} u(x) \int_{\mathbb{R}^{n}\setminus \{0\}} k(x-y,y)\varphi(x-y)-k(x,y)\varphi(x) \nu(dy)dx.
\end{align}
Since $k(x,y)=k(x,-y)$, the adjoint operator $A^{\ast}$ is given by
\begin{align}\label{adjointoperator}
 A^{\ast}\varphi(x) =& -\text{div} [b(x)\varphi(x)] +\int_{\mathbb{R}^{n}\setminus \{0\}}\varphi(x-y)k(x-y,y)-\varphi(x)k(x,y) \nu(dy) \nonumber\\
                = & - \text{div} [b(x)\varphi(x)]+\int_{\mathbb{R}^{n}\setminus \{0\}}(k(x-y,y)-k(x,y))\varphi(x-y)\nu(dy)\nonumber\\
              &+\int_{\mathbb{R}^{n}\setminus \{0\}}(\varphi(x-y)-\varphi(x))k(x,y) \nu(dy) \nonumber\\
                 = & -\text{div} [b(x)\varphi(x)]+ \lim_{\epsilon\searrow 0}\int_{\mathbb{R}^n\setminus \{|x-y|>\epsilon\}}\frac{k(y,x-y)-k(x,x-y)}{|x-y|^{n+\alpha}}\varphi(y)dy \nonumber\\                  &+\int_{\mathbb{R}^{n}\setminus \{0\}}(\varphi(x+y)-\varphi(x))k(x,y) \nu(dy) \nonumber\\
                =& A\varphi(x)-2 b(x)\cdot\nabla \varphi(x) -\text{div}(b(x))\varphi(x) + \mathcal{S}\varphi(x),
\end{align}
where $\mathcal{S}$ is a singular intergal operator defined by
\begin{equation}\label{singular}
    \mathcal{S}\varphi(x)=\lim_{\epsilon\searrow 0}\int_{ \{|x-y|>\epsilon\}} \mathcal{K}(x,x-y)\varphi(y)dy,
\end{equation}
with the kernel
\begin{equation*}
    \mathcal{K}(x,y)=\frac{k(x-y,y)-k(x,y)}{|y|^{n+\alpha}}=\frac{\det(\sigma^{-1}(x-y))}{|\sigma^{-1}(x-y)y|^{n+\alpha}}-\frac{\det(\sigma^{-1}(x))}{|\sigma^{-1}(x)y|^{n+\alpha}}.
\end{equation*}

\begin{remark}
Since $k(x,-y)=k(x,y)$, we have
\begin{align*}
 A^{\ast}\varphi(x) =& -\text{div} [b(x)\varphi(x)] +\int_{\mathbb{R}^{n}\setminus \{0\}}\varphi(x+y)k(x+y,y)-\varphi(x)k(x,y) \nu(dy) .
\end{align*}
When $n=1$, function $k(x,y)=|\sigma(x)|^{\alpha}$. So
\begin{equation*}
  A^{\ast}\varphi(x)=-\text{div} [b(x)\varphi(x)] - (-\Delta)^{\frac{\alpha}{2}}(|\sigma(x)|^{\alpha}\varphi(x)).
\end{equation*}
It is consistent with the form of Fokker-Planck equation which was given in \cite{SLDY2000}. The authors also establish the existence and uniqueness of weak solutions to Fokker-Planck equations when $b$ and $\sigma$ are Lipschitz in \cite{SLDY2000}.
\end{remark}

\subsection{Regularity of invariant measure}
The probability density of invariant measure $p_{ss}(x)$ satisfies the nonlocal elliptic equation $A^{\ast}p_{ss}(x)=0$ in weak sense:
\begin{equation*}
    \int_{\mathbb{R}^n} A\varphi(x)p_{ss}(x)dx=0, \quad \forall \varphi\in C^{\infty}_0(\mathbb{R}^n).
\end{equation*}
But in order to make sure that the conjugate variable which given in next section is well-defined, we need higher regularity of $p_{ss}$.
\begin{lemma}\label{singularoperator}
Assume that $({\bf A})$-$({\bf C})$ hold. Then\\
(i) Singular integral operator $\mathcal{S}$ is a bounded linear operator from $L^p(\mathbb{R}^n)$ to $L^p(\mathbb{R}^n)$, for all $1<p<\infty $.\\
(ii) $\mathcal{S}^{\ast}=-\mathcal{S}$ in $L^p(\mathbb{R}^n)$ with $1<p<\infty$.
\end{lemma}
\begin{proof}
(i) By Calder\'{o}n-Zygmund theory, it suffices to show that the kernel $\mathcal{K}$ satisfies
\begin{equation*}
  |\nabla_y \mathcal{K}(x,y)| \leq C |y|^{-n-1} \quad \text{for almost every } x\in \mathbb{R}^n, y\neq 0.
\end{equation*}
From $(\bf{B})$ and $(\bf{C})$, for almost every $x\in \mathbb{R}^n, y\in \mathbb{R}^{n}\setminus \{0\}$, we have
\begin{align*}
  |\nabla_y \mathcal{K}(x,y)| \leq  & \frac{|\nabla_y\det(\sigma(x-y))|}{|\det(\sigma(x))|^2|\sigma^{-1}(x-y)y|^{n+\alpha}}+(n+\alpha)\frac{|\det(\sigma^{-1}(x-y))|}{|\sigma^{-1}(x-y)y|^{n+\alpha+1}}|\nabla_y (\sigma^{-1}(x-y)y)| \\
   & + (n+\alpha)\frac{|\det(\sigma^{-1}(x))|}{|\sigma^{-1}(x)y|^{n+\alpha-1}}|\nabla_y (\sigma^{-1}(x)y)|\\
   \leq & C |y|^{-n-1}.
\end{align*}
(ii) From (\ref{adjointoperator}), for each $\varphi(x)\in C^{\infty}_0(\mathbb{R}^n)$, we get
\begin{equation*}
  A^{\ast\ast}\varphi(x)= A^{\ast}\varphi(x)- (2 b(x)\cdot\nabla+\text{div}(b(x))^{\ast}\varphi(x) + \mathcal{S}^{\ast}\varphi(x)=A\varphi(x)+\mathcal{S}^{\ast}\varphi(x)+\mathcal{S}\varphi(x)=A\varphi(x)
\end{equation*}
Thus $\mathcal{S}^{\ast}\varphi(x)=-\mathcal{S}\varphi(x)$ for each $\varphi(x)\in C^{\infty}_0(\mathbb{R}^n)$. Since $L^p(\mathbb{R}^n)$ with $1<p<\infty$ is a reflexive Banach space and $C^{\infty}_0(\mathbb{R}^n)$ is dense in $L^p(\mathbb{R}^n)$, $\mathcal{S}^{\ast}=-\mathcal{S}$ in $L^p(\mathbb{R}^n)$ with $1<p<\infty$.
The proof is complete.
\end{proof}

\begin{lemma}\label{Lpsemigroup}
Assume that $({\bf A})$-$({\bf C})$ hold. Then\\
(i) The semigroup $P^{\ast}_t$ associated with Fokker-Planck equation (\ref{FKequ}) can be extended to a strongly continuous semigroup on $L^p(\mathbb{R}^n)$ for all $1\leq p \leq \infty$ with $\|P^{\ast}_t\|_{L(L^p)} \leq C$ for some constant $C>0$;\\
(ii) The resolvent set $\rho(A^{\ast})\supset (0,\infty)$, and $\|(\lambda I-A^{\ast})^{-1}\|_{L(L^p)}\leq \frac{C}{\lambda}$ for all $\lambda >0$.
\end{lemma}

\begin{proof}
By theorem \ref{ExD}, for every $f \in L^1(\mathbb{R}^n)$, let $p_0=|f|/\|f\|_{L^1}$. Then $p_0$ is a probability density, and the probability density at $t$ for all $t>0$ is given by
\begin{equation*}
    p_t(y)=\int_{\mathbb{R}^n}p_0(x)p(0,x;t,y)dx.
\end{equation*}
Note that have $P^{\ast}_t f = \|f\|_{L^1}p_t$. From the definition of transition probability, $p_t \in L^1(\mathbb{R}^n)$ and
\begin{equation*}
    \|p_t\|_{L^1}=\|p_0\|_{L^1}, \quad t\geq 0.
\end{equation*}
So for $\forall t>0$, $\|P^{\ast}_t f\|_{L^1}=\|f\|_{L^1}$, and $P^{\ast}_t$ is a strongly continuous contraction semigroup on $L^1(\mathbb{R}^n)$.

The two-side estimate in theorem \ref{tpdE} shows that for each $f \in L^{\infty}(\mathbb{R}^n)$,
\begin{equation*}
    \|P^{\ast}_t f\|_{L^{\infty}}\leq C t^{-\frac{n}{\alpha}}\|f\|_{L^{\infty}}, \quad \forall t>0.
\end{equation*}
for some constant $C>0$. Moreover, we have
\begin{equation*}
    \lim_{t\rightarrow 0}P^{\ast}_t f =f \quad \text{in} \ L^{\infty}(\mathbb{R}^n).
\end{equation*}
So $P^{\ast}_t$ is a strongly continuous semigroup on $L^{\infty}(\mathbb{R}^n)$ with $\|P^{\ast}_t\|_{L(L^{\infty})}\leq C$.

By interpolation inequality, for all $f\in L^{1}\cap L^{\infty}$, $1<p<\infty$
\begin{equation*}
    \|P^{\ast}_t f\|_{L^p} \leq  Ct^{-\frac{n}{\alpha p'}}\|f\|_{L^p}, \quad \forall t>0,
\end{equation*}
where $\frac{1}{p}+\frac{1}{p'}=1$. So $P^{\ast}_t$ is a strong strongly continuous semigroup on $L^p$ with $\|P^{\ast}_t\|_{L(L^p)}\leq C$.

The Hille-Yosida-Phillips Theorem implies that the resolvent set $\rho(A^{\ast})\supset (0,\infty)$. Moreover, $\|(\lambda I-A^{\ast})^{-1}\|_{L(L^p)}\leq \frac{C}{\lambda}$ for all $\lambda >0$.
\end{proof}

Now we consider the weighted $L^1$ space
\begin{equation*}
 \mathbf{L}^1 = \{ f\in L^1 | \int_{\mathbb{R}^n}|x||f(x)| dx < \infty \} ,
\end{equation*}
which equipped with following $\mathbf{L}^1 $ norm
\begin{equation*}
    \|f\|_{\mathbf{L}^1} =\||x|f\|_{L^1}= \int_{\mathbb{R}^n}|f(x)||x| dx.
\end{equation*}
Let $\mathcal{L}^p=L^p \cap \mathbf{L}^1$, $1\leq p \leq \infty$ be a Banach space equipped with norm
\begin{equation*}
    \|f\|_{\mathcal{L}^p}=\|f\|_{\mathbf{L}^1}+\|f\|_{L^p}.
\end{equation*}
We study the adjoint semigroup $P^{\ast}_t$ on $ \mathcal{L}^p$ in the following Lemma.

\begin{lemma}\label{WL2}
Assume that $({\bf A})$-$({\bf D})$ hold. Then\\
(i) The semigroup $P^{\ast}_t$ is a strongly continuous semigroup on $\mathcal{L}^p$ for all $1\leq p \leq \infty$ with $\|P^{\ast}_t\|_{L(\mathcal{L}^p)} \leq C$ for some constant $C\geq 1$.\\
(ii) The resolvent set $\rho(A^{\ast})\supset (0,\infty)$, and $\|(\lambda I-A^{\ast})^{-1}\|_{L(\mathcal{L}^p)}\leq \frac{C}{\lambda}$ for all $\lambda >0$.
\end{lemma}
\begin{proof}
For each $f\in \mathcal{L}^p$, we have $|x|f\in L^1$. Let $p_0 =|f|/\|f\|_{L^1}$, and $X_0$ is a random variable on $(\Omega, \mathcal{F}, \mathbb{P})$ with density $p_0$. Consider the SDE (\ref{Lequ}) with $F(t)=0$. Then from Lemma (\ref{GEME}), $\mathbb{E}|X_t|<\infty$ for all $t>0$. Thus
\begin{equation*}
     \|P^{\ast}_t f\|_{\mathbf{L}^1}=\||x|(P^{\ast}_t f)\|_{L^1}\leq C\||x|f\|_{L^1}=C\| f\|_{\mathbf{L}^1}, \quad \forall t>0.
\end{equation*}
Moreover, from Lemma \ref{Lpsemigroup}, we have
\begin{equation*}
    \|P^{\ast}_t f\|_{L^p}\leq C\|f\|_{L^p}, \quad \forall t>0.
\end{equation*}
Thus the semigroup $P^{\ast}_t$ is a strongly continuous semigroup on $\mathcal{L}^p$ for all $1\leq p \leq \infty$ and
\begin{equation*}
    \|P^{\ast}_t f\|_{\mathcal{L}^p} \leq C\|f\|_{\mathcal{L}^p}, \quad \forall t>0.
\end{equation*}
where $C\geq 1$ is a constant. So (i) is proved.
The Hille-Yosida-Phillips Theorem implies the result in (ii).
\end{proof}

\begin{lemma}\label{3434}
Assume that $({\bf A})$-$({\bf D})$ hold. Then \\
(i) The density $p_{ss}$ is positive, i.e. $p_{ss}(x)>0$ for all $x\in \mathbb{R}^n$;\\
(ii) For each $1\leq p \leq \infty$, the density of the invariant measure $p_{ss}\in L^p(\mathbb{R}^n)$ and $p_{ss}\in \mathcal{L}^p$.
\end{lemma}
\begin{proof}
Lemma \ref{density} shows that the density of invariant measure $p_{ss} \in L^{p'}(\mathbb{R}^n)$, with $p'<\frac{n}{n-\alpha}$ for $n\geq 2$ and $ p'<\infty$ for $n=1$.  In addition, from definition of invariant measure, the density of invariant measure $p_{ss}$ satisfies
\begin{equation*}
    p_{ss}(y)=\int_{\mathbb{R}^n}p_{ss}(x)p(0,x;t,y)dx
\end{equation*}
for all $t>0$. The two-sides estimates implies that $p(0,x;t,y)>0$. Thus we obtain that $p_{ss}(x)>0$ for each $x\in\mathbb{R}^n$, and $p_{ss}(x)\in L^{\infty}(\mathbb{R}^n)$. Then the interpolation inequality implies that $p_{ss} \in L^p(\mathbb{R}^n)$ for all $1\leq p\leq \infty$.
From Lemma \ref{1orderM}, $\int_{\mathbb{R}^n}|x|p_{ss}(x)dx < \infty$. Thus we have $p_{ss} \in \mathcal{L}^p$ for all $1\leq p \leq \infty$.
\end{proof}

Now we consider the following nonlocal elliptic equation
\begin{equation}\label{nonlocalAE}
  A^{\ast}u(x)-\lambda u(x)=f(x), \quad x\in \mathbb{R}^n,
\end{equation}
where $A^{\ast}$ is defined as (\ref{adjointoperator}), $f\in L^p(\mathbb{R}^n)$ for some $1< p< \infty$, and $\lambda \geq 0$.
The solvability and a prior estimate for nonlocal elliptic equations (\ref{nonlocalAE}) are given as following.

\begin{lemma}\label{3535}
Assume that $({\bf A})$-$({\bf C})$ hold. For some $\lambda_1 \geq 1$ large enough and for all $\mu \geq \lambda_1$, $\vartheta \in [0,\alpha)$, and for $f\in L^p(\mathbb{R}^n)$,  there exists a unique solution $u\in H^{\vartheta}_p(\mathbb{R}^n)$ to the following nonlocal elliptic equation:
\begin{equation}\label{nonlocaleq}
    Au-2b(x)\cdot\nabla u(x)-\mu u(x)=f(x), \quad x\in \mathbb{R}^n.
\end{equation}
Moreover, there is a positive constant $N$, independent of $u$, such that
\begin{equation}\label{apriori}
    \|u\|_{H^{\vartheta}_p} \leq N\|f\|_{L^p}.
\end{equation}
\end{lemma}
\begin{proof}
If $f\in C^{\infty}_c(\mathbb{R}^n)$, then we can obtain a unique smooth solution $u$ for (\ref{nonlocaleq}) by
\begin{equation*}
  u(x)=-\mathbb{E}_x \int^{\infty}_{0}e^{-\mu t} f(Y_t)dt,
\end{equation*}
where $(Y_t)_{t\geq 0}$ is the Markov process associated to the operator $\hat{A}u=(A-2b(x)\cdot\nabla) u$. Now we show the a priori estimate (\ref{apriori}). Suppose $u\in H^{\vartheta}_p(\mathbb{R}^n)$ satisfies (\ref{nonlocaleq}). Let $T>0$ and $\phi(t)$ be a nonnegative and nonzero smooth function with support $(0,T)$. Let $\hat{u}(t,x)=\phi(t)u(x)$. Then
\begin{equation*}
    \partial_t \hat{u}= A\hat{u}-2b(x)\cdot\nabla \hat{u}-\mu \hat{u} +(u\phi'-f\phi ).
\end{equation*}
By Lemma \ref{Backe}, we have
\begin{equation*}
    \|u\|_{H^{\vartheta}_p}\|\phi\|_{L^{\infty}}\leq  \frac{c}{\mu} (\|u\|_{L^p}\|\phi'\|_{L^{\infty}}+\|f\|_{L^p}\|\phi\|_{L^{\infty}}).
\end{equation*}
Letting $\mu$ be large enough, we get the a priori estimate
\begin{equation}
    \|u\|_{H^{\vartheta}_p} \leq N\|f\|_{L^p}.
\end{equation}
Then by a dense argument, we get the unique solution $u\in H^{\vartheta}_p(\mathbb{R}^n)$ of (\ref{nonlocaleq}). The result follows.
\end{proof}

Now we consider the weak solution of nonlocal elliptic equation (\ref{nonlocalAE}).
A function $u\in L^p$ is called weak solution to the nonlocal elliptic equation (\ref{nonlocalAE}) if
\begin{equation*}
  \int_{\mathbb{R}^n}u(x)(A\varphi(x)+\lambda\varphi(x))dx=\int_{\mathbb{R}^n}f(x)\varphi(x)dx, \quad \forall\varphi \in C^{\infty}_0(\mathbb{R}^n).
\end{equation*}
We now state the following a priori estimate of weak solution.

\begin{theorem}\label{Regularitynonlocal}
Assume that conditions $({\bf A})$-$({\bf D})$ hold. Let $f\in L^p(\mathbb{R}^n)$ with some $1<p<\infty$, and let $u\in L^p(\mathbb{R}^n)$ be a weak solution to the nonlocal elliptic equation (\ref{nonlocalAE}). Then $u \in H^{\vartheta}_{p}(\mathbb{R}^n)$ for all $\vartheta \in [0,\alpha)$.
\end{theorem}
\begin{proof}
Since $u$ is a weak solution to the nonlocal elliptic equation $(A^{\ast}-\lambda)u(x)=f(x)$, it satisfies
\begin{equation}\label{weaksol}
 \int_{\mathbb{R}^n}u(x) A\varphi(x)-\lambda u(x) \varphi(x) dx = \int_{\mathbb{R}^n}f(x)\varphi(x) dx, \quad \forall \varphi \in C^{\infty}_0(\mathbb{R}^n).
\end{equation}
From (\ref{adjointoperator}), we have
\begin{equation*}
  (A-2b(x)\nabla)^{\ast}\varphi(x)=A\varphi(x)+\text{div}(b(x))\varphi(x) + \mathcal{S}\varphi(x)
\end{equation*}
From Lemma \ref{singularoperator}, $\mathcal{S}^{\ast}=-\mathcal{S}$ in $L^p(\mathbb{R}^n)$ with $1<p<\infty$. So (\ref{weaksol}) is equivalent to
\begin{equation*}
    \int_{\mathbb{R}^n}u(x) (A-2b(x)\nabla)^{\ast}\varphi(x)-\mu u(x) \varphi(x) dx = \int_{\mathbb{R}^n}g(x)\varphi(x) dx, \quad \forall \varphi \in C^{\infty}_0(\mathbb{R}^n),
\end{equation*}
where $\mu \geq \lambda_1$, $g(x)=\text{div}(b(x))u(x)-\mathcal{S}u(x)-(\lambda+\mu) u(x)+f(x)$.

The nonlocal elliptic equation $(A^{\ast}-\lambda)u=f$ can be rewritten as
\begin{equation*}
    Au(x)-2b(x)\cdot\nabla u(x)-\mu u(Rx)= g(x), \quad x \in\mathbb{R}^n,
\end{equation*}
where $\mu \geq \lambda_1$.
From assumption (${\bf C}$) and Lemma \ref{singularoperator}, $F:u\rightarrow \text{div}(b(x))u(x)-\mathcal{S}u(x)$ is a linear bounded mapping from $L^p(\mathbb{R}^n)$ to $L^p(\mathbb{R}^n)$ for all $1<p<\infty$. Thus $g(x)\in L^p(\mathbb{R}^n)$ for all $1<p<\infty$.
Then by Lemma \ref{3535}, there exists a unique solution $\hat{u}\in H^{\vartheta}_p$ for each $\vartheta \in [0,\alpha)$ to the following equation
\begin{equation}\label{equlll}
   (A -2b(x)\cdot\nabla )\hat{u}(x)-\mu \hat{u}(x)=g(x), \quad x\in \mathbb{R}^n,
\end{equation}
where $g(x)=F(u)(x)-(\lambda+\mu) u(x)+f(x) \in L^p(\mathbb{R}^n)$.
Moreover, $\hat{u}$ satisfies following identity
\begin{equation*}
    \int_{\mathbb{R}^n}\hat{u}(x) (A-2b(x)\nabla)^{\ast}\varphi(x)-\mu \hat{u}(x) \varphi(x) dx = \int_{\mathbb{R}^n}g(x)\varphi(x) dx, \quad \forall \varphi \in C^{\infty}_0(\mathbb{R}^n).
\end{equation*}
Now we show that $u=\hat{u}$. It is sufficient to show that the weak solution of equation (\ref{equlll}) is unique. Since the operator $\hat{A}:=A-2b(x)\nabla$ is also a generater of Markov semigroup $\hat{P}_t$ which satisfies $\bf{(A)}$-$\bf{(C)}$. Then by Lemma \ref{Lpsemigroup}, for $\forall \mu >\lambda_1$, the resolvent operator $\mu I-\hat{A}^{\ast}$ is a bijective operator on $L^p(\mathbb{R}^n)$ with $1\leq p \leq \infty$.
So the set
\begin{equation*}
  \{\mu I\varphi -(A-2b(x)\nabla)^{\ast}\varphi : \forall \varphi\in C^{\infty}_0(\mathbb{R}^n)\}
\end{equation*}
is dense in $L^p(\mathbb{R}^n)$ for all $1\leq p \leq \infty$ and $\mu >\lambda_1 $. It implies that the weak solution of (\ref{equlll}) is unique, and $\hat{u} =u$. Thus $u \in H^{\vartheta}_p(\mathbb{R}^n)$ for all $\vartheta\in [0,\alpha)$.
\end{proof}

Now we have the following regularity result for the density of the invariant measure.

\begin{theorem}\label{regularity}
Assume that conditions $({\bf A})$-$({\bf D})$ hold. Then the unique invariant measure has a density $p_{ss} \in H^{\vartheta}_{p}(\mathbb{R}^n)$ for all $\vartheta \in [0,\alpha)$ and $1 < p < \infty$.
\end{theorem}
\begin{proof}
  By the definition of the invariant measure, $p_{ss}$ is the weak solution of nonlocal elliptic equation $A^{\ast}p_{ss}=0$, i.e.
 \begin{equation*}
   \int_{\mathbb{R}^n} p_{ss}(x) A \varphi (x) dx, \quad  \forall \varphi\in C^{\infty}_0(\mathbb{R}^n).
 \end{equation*}
From Lemma \ref{3434}, $p_{ss} \in L^p(\mathbb{R}^n)$ for all $1\leq  p \leq \infty$, Then the regularity result in Theorem \ref{Regularitynonlocal} implies that $p_{ss} \in H^{\vartheta}_p(\mathbb{R}^n)$ for all $1< p <\infty$ and $\vartheta \in [0,\alpha)$.
\end{proof}

\section{Linear Response Theory}
In this section, we derive the response function. Furthermore, we establish the linear response theory and the Agarwal-type fluctuation-dissipation theorem for SDE (\ref{Lequ}).

\subsection{The response function}
In this subsection, we derive the response function.
We assume that $(X_{t})_{t\geq 0}$ is a stationary Markov process, which satisfies the SDE (\ref{Lequ}), and its initial distribution is the unique invariant measure $\mu$ of the corresponding Markov semigroup $P_t$. Let $(X^{F}_t)_{t\geq 0}$ be the unique strong solution of the perturbed SDE (\ref{LequU}), which is the perturbed process under perturbation $F$. We denote $(P_t^{ F})_{t\geq 0}$ the corresponding Markov semigroup of the perturbed process $(X_t^{F})_{t\geq 0}$. The generator of $(X^{ F})_{t\geq 0}$ is denoted by
\begin{equation*}
    A^{F}u(x)=Au(x)+F(t) K(x)\cdot \nabla u(x):=(A+F(t) L )u(x),
\end{equation*}
where the external perturbation operator $Lu=K(x)\cdot \nabla u$. Then the associated Fokker-Planck equation of the perturbed process $(X^{F}_t)_{t\geq 0}$ is
\begin{equation}\label{UFKequ}
  \left\{
   \begin{aligned}
   & \partial_t p_t^{ F}(x) = A^{\ast} p_t^{ F}(x)+ F(t) L^{\ast} p_t^{ F}(x),\\
   & p_0^{F}(x)=p_{ss}(x),
   \end{aligned}
   \right.
\end{equation}
where $ L^{\ast} u(x)=- \text{div}(K(x)u(x))$, and $p^F_t$ is the probability density of $X^F_t$.

Now we redefine the response function of an observable in (\ref{resf}) mathematically. This definition means that when the system closed to the steady state, the change in the expectation value of every observable is linear with the small perturbing source.

\begin{definition}
Let $O(x)\in L^p(\mathbb{R}^n)$ be an observable for some $1< p \leq \infty$. Let $(X_{t})_{t\geq 0}$ be a stationary Markov process. For every $\phi\in L^{\infty}(\mathbb{R}^{+})$, let $(X^{\epsilon\phi}_t)_{t\geq 0}$ be the perturbed process under perturbation $\epsilon\phi(t)K(x)$ with initial value $X_0 $. Then a locally integrable function $\mathcal{R}_{O}$ is called the response function of the observable $O$ if it satisfies
\begin{equation*}
\lim_{\epsilon \rightarrow 0}\frac{1}{\epsilon} (\mathbb{E}O(X_t^{\epsilon \phi})-\mathbb{E}O(X_{0}))= \int_0^t \mathcal{R}_{O}(t-s)\phi(s) ds.
\end{equation*}
\end{definition}

The following lemma shows the perturbation property of Markov semigroup $P_t$.

\begin{lemma}\label{43}
Assume that $({\bf A})$-$({\bf D})$ hold. Then for each $f\in L^{p}(\mathbb{R}^n)$ with $1<p\leq \infty$, and $0\leq s\leq t $, we have
\begin{equation*}
    P^{\epsilon \phi}_{s,t} f(x)-P_{s,t} f(x)=\epsilon\int_s^t \phi(r) P_{r-s}(K(x)\cdot \nabla(P^{\epsilon\phi}_{r,t}f( x)))dr.
\end{equation*}
Moreover, for $t>0$,
\begin{equation*}
    \lim_{\epsilon\rightarrow 0}\frac{1}{\epsilon}(P^{\epsilon \phi}_{0,t} f(x)-P_{t} f(x)) = \int_0^t \phi(s) P_{s}(K(x)\cdot \nabla(P_{t-s}f(x)))ds, \quad \text{in} \ L^p(\mathbb{R}^n).
\end{equation*}
\end{lemma}
\begin{proof}
Denote $P^{\epsilon \phi}_{s,t} f(x)=w_1(s,x)$, $P_{s,t}f(x)=w_2(s,x)$, and $w(s,x)=w_1(s,x)-w_2(s,x)$. Then $w_1$ is the solution of the Kolmogorov backward equation
\begin{equation}
  \left\{
   \begin{aligned}
   & \partial_s w_1(s,x) = -A^{\epsilon\phi}(s)w_1(s,x),\quad (s,x)\in [0,t]\times\mathbb{R}^n,\\
   & w_1(t,x)=f(x),\quad x\in \mathbb{R}^n,
   \end{aligned}
   \right.
\end{equation}
and $w_2$ is the solution of the Kolmogorov backward equation
\begin{equation}
  \left\{
   \begin{aligned}
   & \partial_s w_2(s,x) = -Aw_2(s,x),\quad (s,x)\in [0,t]\times\mathbb{R}^n,\\
   & w_2(t,x)=f(x),\quad x\in \mathbb{R}^n.
   \end{aligned}
   \right.
\end{equation}
Lemma \ref{EXuBack} implies that above two equations exist unique $H^{\vartheta}_{p}(\mathbb{R}^n)$ solution for all $\vartheta \in (1,\alpha)$, such that
\begin{equation*}
  \|w_i(s,t)\|_{H^{\vartheta}_{p}} \lesssim (t-s)^{-\frac{\vartheta}{\alpha}}\|f\|_{L^p},\quad i=1,2.
\end{equation*}
Then for some $\beta \in (1,\alpha)$, by Sobolev embedding $H^{\beta}_p \hookrightarrow H^{1}_{p}$,
\begin{equation*}
  \|w_i(s,t)\|_{H^{1}_{p}} \lesssim (t-s)^{-\frac{\beta}{\alpha}}\|f\|_{L^p},\quad i=1,2.
\end{equation*}
Note that $w(s,x)$ satisfies following equation
\begin{equation}\label{wee}
  \left\{
   \begin{aligned}
   & \partial_s w(s,x) = -Aw(s,x)- \epsilon g^{\epsilon}(s,x), \quad (s,x)\in [0,t]\times\mathbb{R}^n,\\
   & w(t,x)=0, \quad x\in \mathbb{R}^n,
   \end{aligned}
   \right.
\end{equation}
where $g^{\epsilon}(s,x):=\phi(s)Lw_1(s,x)=\phi(s)L(P^{\epsilon\phi}_{s,t}f(x))$. From Lemma \ref{EXuBack}, for $0\leq s \leq t$, we have
\begin{equation*}
    \|g^{\epsilon}(s,x)\|_{L^{p}(\mathbb{R}^n)}\leq \|\phi\|_{L^{\infty}[0,t]}\|\nabla w_1(s,x)\|_{L^{p}(\mathbb{R}^n)}\lesssim (t-s)^{-\frac{\beta}{\alpha}}\|\phi\|_{L^{\infty}[0,t]}\|f\|_{L^{p}(\mathbb{R}^n)}.
\end{equation*}
For some $\gamma \in [\frac{\alpha}{\beta},\alpha)$,
\begin{equation*}
    \|g^{\epsilon}\|_{L^{\gamma}(s,t;L^{p})}\lesssim (t-s)^{\frac{1}{\gamma}-\frac{\beta}{\alpha}}\|\phi\|_{L^{\infty}[0,t]}\|f\|_{L^p(\mathbb{R}^n)}\leq \|\phi\|_{L^{\infty}[0,t]}\|f\|_{L^p(\mathbb{R}^n)}.
\end{equation*}

Then by Lemma \ref{Backe}, there is a unique solution $w$ to (\ref{wee}), and it satisfies
\begin{equation*}
   w(s,x)=P^{\epsilon \phi}_{s,t} f(x)-P_{s,t} f(x)=\epsilon\int^t_s P_{s,r}g^{\epsilon}(r,x)dr=\epsilon\int^t_s P_{s,r}\phi(r)L(P^{\epsilon\phi}_{r,t}f(x))dr,
\end{equation*}
and
\begin{align*}
  \|P^{\epsilon \phi}_{r,t} f-P_{r,t} f\|_{L^{\gamma}(s,t;H^1_{p})} & = \|w\|_{L^{\gamma}(s,t;H^1_{p})} \\
   & \lesssim \epsilon \|g^{\epsilon}\|_{L^{\gamma}(s,t;L^{p})}\\
   & \lesssim \epsilon c\|\phi\|_{L^{\infty}[0,t]}\|f\|_{L^p}.
\end{align*}
So
\begin{equation}
    \lim_{\epsilon\rightarrow 0}\|P^{\epsilon \phi}_{r,t} f-P_{r,t} f\|_{L^{\gamma}(0,t;H^1_{p})} = 0.
\end{equation}
This implies that
\begin{equation*}
  \lim_{\epsilon\rightarrow 0}\|g^{\epsilon}(s,x)-\phi(s) L(P_{t-s}f(x))\|_{L^1(0,t;L^{p})}=0
\end{equation*}
Thus for $t>0$, we have
\begin{align*}
\lim_{\epsilon\rightarrow 0}\frac{1}{\epsilon}(P^{\epsilon \phi}_{0,t} f(x)-P_{t} f(x)) & = \lim_{\epsilon\rightarrow 0}\int_0^t g^{\epsilon}(s,x)ds \\
    & =\int_0^t \lim_{\epsilon\rightarrow 0}g^{\epsilon}(s,x)ds\\
    & =\int_0^t \phi(s) P_{s}L(P_{t-s}f(x))ds \quad \text{in} \ L^p(\mathbb{R}^n).
\end{align*}
\end{proof}

Now we state the main results of this subsection.

\begin{theorem}
Assume that $({\bf A})$-$({\bf D})$ hold. Suppose $F(t)\in L^{\infty}(\mathbb{R}^{+})$, and $K(x), \text{div}(K(x)) \in  L^{\infty}(\mathbb{R}^n)$. Let $O(x)\in L^p(\mathbb{R}^n)$ be an observable with some $1< p \leq \infty$. Then the response function $\mathcal{R}_{O}$ is given by
\begin{equation*}
    \mathcal{R}_{O}(t)=\int_{\mathbb{R}^n} L(P_{t}O(x))p_{ss}(x)dx=\int_{\mathbb{R}^n} O(x)  P^{\ast}_{t}(L^{\ast}p_{ss}(x))dx.
\end{equation*}
\end{theorem}

\begin{proof}
For every $t>0$ and $ \phi \in C^{\infty}[0,t]$, we have
\begin{equation*}
  \mathbb{E}O(X^{\epsilon\phi}_t)= \int_{\mathbb{R}^n}\mathbb{E}_x O(X^{\epsilon\phi}_t)p_{ss}(x)dx=\mathbb{E}P^{\epsilon \phi}_{0,t} O(X_{0}),
\end{equation*}
and
\begin{equation*}
  \mathbb{E}O(X_t)= \int_{\mathbb{R}^n}\mathbb{E}_x O(X_t)p_{ss}(x)dx=\mathbb{E}P_{0,t} O(X_{0}).
\end{equation*}
Then it follows from Lemma \ref{43} that
\begin{align*}
\lim_{\epsilon \rightarrow 0}\frac{1}{\epsilon} (\mathbb{E}O(X_t^{\epsilon \phi})-\mathbb{E}O(X_{0})) = &  \lim_{\epsilon \rightarrow 0}\mathbb{E}\frac{1}{\epsilon} (P^{\epsilon \phi}_{0,t} O(X_{0})-P_{0,t} O(X_{0})) \\
 =& \int_0^t \phi(s) P_{s}L(P_{t-s}f(X_{0}))ds\\
 =& \int_0^t\int_{\mathbb{R}^n} \phi(s)P_{s}L(P_{t-s}O(x))dsp_{ss}(x)dxds \\
 =& \int_0^t \phi(s)\int_{\mathbb{R}^n} L(P_{t-s}O(x))p_{ss}(x)dxds
\end{align*}
Thus
\begin{align*}
     \mathcal{R}_{O}(t) =\int_{\mathbb{R}^n} L(P_{t}O(x))p_{ss}(x)dx =\int_{\mathbb{R}^n} O(x)  P^{\ast}_{t}(- \text{div}(K(x)p_{ss}(x))))dx.
\end{align*}
The proof is complete.
\end{proof}

\subsection{The conjugate variable}
In nonequilibrium statistical mechanics, the fluctuation-dissipation theorem reveals the response of an observable physical quantity to a small external perturbation, by the correlation function of this observable physical quantity and another observable physical quantity that is a conjugate variable to the perturbation with respect to energy.

First, we have the following Agarwal-type fluctuation dissipation theorem(see \cite{P2016}).
\begin{theorem}\label{AFDT}
(Agarwal-type fluctuation dissipation theorem) Assume that $({\bf A})$-$({\bf D})$ hold, and $F(t)\in L^{\infty}(\mathbb{R}^{+})$, $ K(x), \text{div}(K(x)) \in  L^{\infty}(\mathbb{R}^n)$. Then for every observable $O(x) \in L^p(\mathbb{R}^n)$ for some $1\leq p \leq \infty$, there exists another observable $Y$ defined as
\begin{equation*}
    Y(x) =\frac{- \text{div}(K(x)p_{ss}(x))}{p_{ss}(x)},
\end{equation*}
such that
\begin{equation*}
    \mathcal{R}_{O}(t)=\mathbb{E}( O(X_t)Y(X_{0})).
\end{equation*}
\end{theorem}

\begin{proof}
From Theorem \ref{regularity} and Lemma \ref{3434}, we see that the density $p_{ss} >0$ and $p_{ss} \in H^{\vartheta}_p(\mathbb{R}^n)$ for all $\vartheta\in[0,\alpha)$ and $1< p < \infty$. So the observable $Y$ is well-defined.

Then the cross correlation function with the invariant measure satisfies
\begin{align*}
   \mathbb{E}(O(X_t)Y(X_{0})) = & \mathbb{E}((P_t O(X_{0}))Y(X_{ss}))\\
   =&\int_{\mathbb{R}^n}(P_t O(x))Y(x)p_{ss}(x)dx\\
   = & \int_{\mathbb{R}^n} O(x)P^{\ast}_t (Y(x)p_{ss}(x))dx \\
   = & \int_{\mathbb{R}^n} O(x)  P^{\ast}_{t}(L^{\ast}p_{ss}(x))dx \\
   = & \mathcal{R}_{O}(t).
\end{align*}
The proof is complete.
\end{proof}

Now we define the conjugate variable to the perturbation, and provide the linear response theory of SDE (\ref{Lequ}). Suppose there is a perturbation $\epsilon K(x)$ applied on the drift term, and the perturbed process has a unique invariant measure with density $p_{ss}^{\epsilon}$. From \cite{SS2010,PJP2009}, the conjugate variable $U(x)$ is given by
\begin{equation}\label{CVnotime}
    U(x)=-\left.\frac{ \partial\log p_{ss}^{\epsilon}}{\partial \epsilon}\right|_{\epsilon=0} = \left.\frac{\Phi^{\epsilon}(x)}{\partial \epsilon}\right|_{\epsilon=0}.
\end{equation}
In this deﬁnition, $\Phi^{\epsilon}(x)= \log p_{ss}^{\epsilon}(x)$ stands for a nonequilibrium potential, or stochastic entropy \cite{HS2001,PJP2009}. In \cite{CJ2020}, the authors show that if the conjugate variable $U(x)$ exists, then it has following form
\begin{equation*}
    U(x)=\frac{v(x)}{p_{ss}(x)},
\end{equation*}
where $v(x)$ is a solution of the following elliptic equation
\begin{equation}\label{nonEl}
    A^{\ast}v(x)=L^{\ast}p_{ss}(x).
\end{equation}

Since our perturbation depends on time $t$, it is difficult to define the conjugate variable as in (\ref{CVnotime}) directly. But motivated by above necessary condition, we can define the conjugate variable as $U(x)=v(x)/p_{ss}(x)$, where $v(x)$ is a solution of (\ref{nonEl}).
Before defining the conjugate variable, we need to prove the existence of nonlocal elliptic equation (\ref{nonEl}).
Our approach to investigate the solvability of (\ref{nonEl}) is the Fredholm alternative theorem. Before proving it, we first recall the following compactness result,  which can be found in \cite{RS1978}, Theorem XIII.67.

\begin{definition}\label{CVdef}
The space $\mathcal{H}$ is defined by $\mathcal{H}=\{u\in H^1(\mathbb{R}^n):\int_{\mathbb{R}^n}|x||u|^2dx < \infty \}$ with norm
\begin{equation*}
    \|u\|_{\mathcal{H}}= \|\nabla u\|_{L^2}+\||x|u\|_{L^2}.
\end{equation*}
\end{definition}

\begin{lemma}\label{conpact}
The Sobolev embedding $\mathcal{H}(\mathbb{R}^n)\hookrightarrow L^p(\mathbb{R}^n)$ is compact for $2 \leq p < 2^{\ast}$, where
\begin{equation*}
  2^{\ast}:=\left\{
   \begin{aligned}
   & \frac{2n}{n-2},\quad 1\leq 2 <n\\
   & \infty,\quad 2\geq n
   \end{aligned}
   \right.
\end{equation*}
is the Sobolev conjugate of $2$.
\end{lemma}

Applying the above compactness result, we are now ready to establish the following solvability result.

\begin{lemma}\label{LLPP}
Assume that conditions $({\bf A})$-$({\bf D})$ hold. Suppose $f \in \mathcal{L}^p$ for some $2\vee(\frac{n}{\alpha}+1)<p< \infty$, and it satisfies $\int_{\mathbb{R}^n}f(x)dx\neq 0$. Then there exists a nonzero solution $u \in H^{\vartheta}_p$ for all $\vartheta \in [0,\alpha)$ to the equation
\begin{equation}\label{nonlocalff}
  A^{\ast}u(x)=f(x), \quad x\in \mathbb{R}^n.
\end{equation}
And there exists a nonzero solution $u_{ss} \in H^{\vartheta}_p$ for all $\vartheta \in [0,\alpha)$ and all $2\vee(\frac{n}{\alpha}+1)<p< \infty$ to the equation
\begin{equation}\label{nonlocal00}
  A^{\ast}u_{ss}(x)=0, \quad x\in \mathbb{R}^n.
\end{equation}
Moreover, if $u_1$, $u_2$ are both nonzero solutions of the equation (\ref{nonlocalff}), then $u_1 -u_2 =cu_{ss}$, where $c$ is a constant.
\end{lemma}

\begin{proof}\label{Exinh}
From Lemma \ref{WL2}, the inverse operator $T=(I-A^{\ast})^{-1}$ on $\mathcal{L}^p$ exists with $1\leq p \leq \infty$. Moreover, Theorem \ref{Regularitynonlocal} implies that $T$ is a linear bounded operator from $\mathcal{L}^p$ to $H^{\vartheta}_p\cap\mathcal{L}^p $ for all $\vartheta \in [0,\alpha)$.

Then the nonlocal elliptic equation $A^{\ast}u =f$ can be rewritten as
\begin{equation}\label{Fredholm}
    (I-T)u=h,
\end{equation}
where $h:=-Tf\in \mathcal{L}^p$ with $2\vee(\frac{n}{\alpha}+1)<p< \infty$.
We now claim that $T:\mathcal{L}^p\rightarrow \mathcal{L}^p$ is a bounded, linear, compact operator. By Sobolev embedding, $H^{\vartheta}_p(\mathbb{R}^n) \hookrightarrow H^1(\mathbb{R}^n)$ for some $\vartheta \in (1,\alpha)$. Moreover, since $p >\frac{n}{\alpha}+1$, we have $H^{\vartheta}_p(\mathbb{R}^n) \hookrightarrow L^{\infty}(\mathbb{R}^n)$ for some $\vartheta \in (\frac{n\alpha}{n+\alpha},\alpha)$. Then the H\"{o}lder inequality implies that
\begin{equation*}
    \int_{\mathbb{R}^n}|x||u(x)|^2dx \leq \|u\|_{L^{\infty}}\||x|u\|_{L^1}\leq C\|u\|_{\mathcal{L}^p}\|u\|_{H^{\vartheta}_p}.
\end{equation*}
Thus $(H^{\vartheta}_p \cap \mathcal{L}^p ) \subset \mathcal{H}$ with $\vartheta \in (1\vee \frac{n\alpha}{n+\alpha},\alpha)$, and
\begin{equation*}
    \|u\|_{\mathcal{H}} \leq C(\|u\|_{\mathcal{L}^p}+\|u\|_{H^{\vartheta}_p}).
\end{equation*}
From Lemma \ref{conpact}, $T:\mathcal{L}^p\rightarrow \mathcal{L}^p$ is a compact operator. So Fredholm alternative holds for the equation $(I-T)u=h$.

From Lemma \ref{3434}, the density of the unique invariant measure $p_{ss}\in \mathcal{L}^p$ with $1\leq p \leq \infty$.. It implies that the equation $(I-T)u=0$ has a nonzero solution $p_{ss} \in \mathcal{L}^p$. Then by Fredholm alternative theorem, $dim N(I-T)=dim N(I-T^{\ast})>0$. So when $f=0$, then $h=-Tf=0$ and there exists a nonzero solution $u\in N(I-T)$. Furthermore, for every $f \in \mathcal{L}^p$ with $f\neq 0$, the equation $A^{\ast}u=f$ has a solution $u_{ss} \in H^{\vartheta}_p\cap\mathcal{L}^p $ if and only if $(f,v)=\int_{\mathbb{R}^n}f(x)v(x) dx=0$ for all $v\in N(I-T^{\ast})$, where $T^{\ast}$ is a bounded linear operator from $(\mathcal{L}^p)^{\ast}$ to $(\mathcal{L}^p)^{\ast}$.

Now we describe the subspace $N(I-T^{\ast})$. By definition of $\mathcal{L}^p$, the smooth bounded function space $C^{\infty}_b(\mathbb{R}^n)$ is a dense subspace of $(\mathcal{L}^p)^{\ast}$. If $w\in C^{\infty}_b(\mathbb{R}^n)$ satisfies equation $(I-T^{\ast})w=0$, then $Aw=0$. By maximum principle of nonlocal elliptic operator $A$, if $w\in C^{\infty}_b(\mathbb{R}^n)$ is a solution to the equation $Aw=0$, then the solution $w(x)$ is a constant. Note that for every constant $c$, $w(x)=c$ is a solution of $Aw=0$, and all constant functions constitute a $1$-dimension linear subspace of $C^{\infty}_b(\mathbb{R}^n)$. We obtain that $N(I-T^{\ast})=\{f=c | c\in \mathbb{R} \ \text{is a constant}\}$, and the above condition $(f,v)=0$ for all $v\in N(I-T^{\ast})$ holds if and only if  $\int_{\mathbb{R}^n}f(x)dx\neq 0 $. Moreover, $dim N(I-T)=dim N(I-T^{\ast})=1$, thus $N(I-T)=\text{span}\{u_{ss}\}$. So if $u_1$, $u_2$ are both nonzero solutions of above equation, then $u_1 -u_2 =cu_{ss}$, where $c$ is a constant. We complete the proof.
\end{proof}

Now we define the conjugate variable of perturbation $F(t)K(x)$ as following:
\begin{definition}\label{CVCV}
The conjugate variable $U(x)$ of perturbation $F(t)K(x)$ is defined by
\begin{equation*}
    U(x)=\frac{v(x)}{p_{ss}(x)},
\end{equation*}
where $v(x)$ is a solution of following nonlocal elliptic equation
\begin{equation*}
    A^{\ast}v(x)=L^{\ast}p_{ss}(x).
\end{equation*}
\end{definition}

Now we state the linear response theory for SDE (\ref{Lequ}).
\begin{theorem}\label{LRThm}
(Linear response theory) Assume that $({\bf A})$-$({\bf D})$ hold, $F(t)\in L^{\infty}(\mathbb{R}^{+})$, $\text{div}(K(x)) \in  L^{\infty}(\mathbb{R}^n)$, $|x|K(x) \in L^{\infty}(\mathbb{R}^n)$, and $\int_{\mathbb{R}^n}\text{div}(K(x)p_{ss}(x)) dx \neq 0$. Then for every observable $O(x) \in L^{p}(\mathbb{R}^n)$ for some $1<p\leq\infty$, there exists a conjugate variable $U$ defined as Definition (\ref{CVCV}), such that
\begin{equation*}
    \mathcal{R}_{O}(t)=\frac{d}{dt}\mathbb{E}( O(X_t)U(X_{0})).
\end{equation*}
\end{theorem}
\begin{proof}
By theorem \ref{regularity}, $p_{ss} \in H^{\vartheta}_p(\mathbb{R}^n)$ for all $1<p<\infty$ and $\vartheta\in [0,\alpha)$. By $K(x)\in H^1_{\infty}(\mathbb{R}^n)$ and Sobolev embedding, $L^{\ast}p_{ss}=-\text{div}(K(x)p_{ss}(x))\in L^p(\mathbb{R}^n)$ with $1\leq p \leq \infty$. Moreover, by Lemma \ref{3434}, $p_{ss} \in\mathcal{L}^p$ with $1\leq p \leq \infty$. Combing with $|x|K(x)\in L^{\infty}(\mathbb{R}^n)$, we have $L^{\ast}p_{ss} \in \mathcal{L}^p$ for all $1<p<\infty$.
Then Lemma \ref{Exinh} implies that there exists a solution $v(x)$ of equation $A^{\ast}v(x)=L^{\ast}p_{ss}(x)$. Since $p_{ss}(x)>0$ for all $x\in \mathbb{R}^n$, the conjugate variable $U(x)=v(x)/p_{ss}(x)$ exists.

From definition of the conjugate variable $U(x)$, we have
\begin{equation*}
    \mathbb{E}(  O(X_t)U(X_{0}))=\int_{\mathbb{R}^n}O(x)P^{\ast}_t v(x)dx.
\end{equation*}
From Lemma \ref{Lpsemigroup}, $P^{\ast}_t$ is a strongly continuous semigroup on $L^p(\mathbb{R}^n)$ for all $1\leq p \leq \infty$. Then by the dominated convergence theorem and definition of generator, we have
\begin{align*}
    \frac{d}{dt}\mathbb{E}(  O(X_t)U(X_{0}))& = \frac{d}{dt}\int_{\mathbb{R}^n}(P_t O(x))U(x)p_{ss}(x)dx\\
    & =\int_{\mathbb{R}^n} O(x)\frac{d}{dt}P^{\ast}_t (v(x))dx\\
     & =\int_{\mathbb{R}^n}O(x)P^{\ast}_t A^{\ast}v(x)dx \\
    & =\int_{\mathbb{R}^n}O(x)P^{\ast}_t (L^{\ast}p_{ss}(x)) dx \\
    & = \mathcal{R}_{O}(t).
\end{align*}
This completes the proof.
\end{proof}

\begin{remark}
The linear response theory for SDE (\ref{Lequ}) is also called Seifert-Speck type fluctuation-dissipation theorem \cite{SS2010}.
\end{remark}

\section{Conclusion}
We have established a linear response theory for the nonlinear stochastic differential equation driven by an $\alpha$-stable L\'{e}vy process ($1<\alpha<2$) under a perturbation $F(t)K(x)$ on the drift term. In addition, we have developed the Agarwal-type fluctuation-dissipation theorem for this stochastic system. Our results show the susceptibility of every observable under an small time-dependent perturbation when the system is close to the steady state.
During the proof,  we prove a new ergodicity results by the Bogoliubov-Krylov argument, the response function was also obtained by investigating the perturbation property of the corresponding Markov semigroup $P_t$. We also have shown existence and regularity results for the stationary Fokker-Planck equations by the a priori estimate.

There are still some limitations of our results. Due to the requirement of solvability and regularity of Kolmogorov backward equations for the corresponding SDEs, we restrict $1<\alpha<2$ in this paper. For the same reason, we only consider bounded drift term $b(x)$. In order to prove the existence of conjugate variable, it is important to assume that $|x|K(x)$ is bounded. We also ask that the perturbation can be written in the form $F(t)K(x)$, and the perturbation only applied on the drift term. These points will be the subjects of future work.

\end{document}